\documentclass[12pt]{amsart}
\usepackage{mathrsfs,eucal,amsmath,amsthm,amsfonts,amssymb,amscd,amsbsy,latexsym,dsfont}
\usepackage[all,2cell,dvips]{xy} \UseAllTwocells \SilentMatrices

\usepackage{verbatim}

\newtheorem{thm}{Theorem}
\newtheorem{dfn}[thm]{Definition}
\newtheorem{lem}[thm]{Lemma}
\newtheorem{cor}[thm]{Corollary}

\newtheorem{prp}[thm]{Proposition}

\begin{document}

\title[Vector bundles trivialized by proper morphisms]{Vector bundles trivialized by 
proper morphisms and the fundamental group scheme}

\author[I. Biswas]{Indranil Biswas}

\address{School of Mathematics, Tata Institute of Fundamental
Research, Homi Bhabha Road, Bombay 400005, India}

\email{indranil@math.tifr.res.in}

\author[J. P. dos Santos]{Jo\~ao Pedro P. dos Santos}

\address{Universit\'e de Paris 6, Institut de Math\'ematiques de Jussieu,
175, Rue du Chevaleret, 75013 Paris, France}

\email{dos-santos@math.jussieu.fr}



\date{}

\begin{abstract}
Let $X$ be a smooth projective variety
defined over an algebraically closed field $k$. 
Nori constructed a category of vector bundles on $X$, called essentially finite vector bundles, which 
is reminiscent of the category of representations of the fundamental 
group (in characteristic zero). In fact, this category is equivalent to the category of representations of a pro--finite group scheme which controls all finite torsors.  We show that essentially finite vector bundles coincide 
with those which become trivial after being pulled back by some proper 
and surjective morphism to $X$.
\end{abstract}

\maketitle

\section{Introduction}\label{first_section}
Let $k$ be an algebraically closed field, and let $X/k$ be a smooth projective 
variety. By a vector bundle we mean a locally free coherent sheaf. 
We are interested in studying vector bundles on $X$ enjoying the 
following property:
\begin{enumerate}
\item[\textbf{(T)}] There exists a proper $k$--scheme $Y$ together with 
a surjective
(proper) morphism $f:Y\longrightarrow X$ such that the pull--back $f^*E$ 
is trivial. 
\end{enumerate}

Note that we can, and usually will, assume that $Y$ is a proper variety: just replace $Y$ by the reduced subscheme underlying an irreducible component of $Y$ dominating $X$. 

In \cite{nori}, Nori introduced the category of
\emph{essentially finite} vector bundles on $X$; this category ``is''  the category of representations of a pro--finite group scheme which generalizes (respectively equals) the \'etale fundamental group of SGA1 in positive characteristic (respectively characteristic zero).  
A vector bundle $E\,\longrightarrow\, X$ is essentially finite if and only if there exists a finite group scheme $G$, a $G$--torsor 
$f\,:\,P\,\longrightarrow\, X$, and a representation $V$ of $G$, such 
that 
$E\,=\, P\times^GV$. (This is not Nori's original definition, but one of the main consequences of \cite{nori}.) Since the $G$--torsor $f^*P$ is canonically 
trivialized, the pull--back $f^*E$ is trivializable, so that essentially finite vector bundles enjoy property (T). 
We establish a converse, which should also be regarded as a generalization of a result 
due to Lange and Stuhler \cite[Proposition 1.2]{lange-stuhler} stating that a vector bundle trivialized by a finite \'etale covering comes from a representation of the \'etale fundamental group.  
The reader can also 
regard what follows in connection with Proposition 
1.2 of [SGA 1, X]; of course, the point here is not to impose conditions
on the trivializing morphism other than properness and surjectivity.

\begin{thm}\label{theorem_main} Let $X$ be a smooth and projective variety over the algebraically closed field $k$. A vector bundle $E$ over $X$ is 
essentially finite if and only if it satisfies property (T).   
\end{thm}

Property (T) of vector bundles is in fact stronger than it appears to be. Here is what we mean by this.  
Let $f:Y\longrightarrow X$ be a proper and surjective morphism and
$E\, \longrightarrow\, X$ a vector bundle such that $f^*E$ is trivial. 
Consider the Stein factorization of $f$
\[
\xymatrix{Y\ar[rr]^-{f'} && Z=\mathrm{Spec}(f_*\mathcal O_Y) 
\ar[rr]^-{g} && X} \, .
\] 
Using the fact that $f_*'\mathcal O_Y=\mathcal O_{Z}$,
it follows that if $V\, \longrightarrow\, Z$ is a vector bundle such 
that $(f')^*V$ is trivial, then $V$  itself
is trivial. Hence, $g^*E$ is trivial. This shows that a vector bundle $E$ 
over $X$ satisfies property (T) if and only if it satisfies the following condition:
\begin{enumerate}
\item[\textbf{(TF)}] There exists a finite and surjective morphism $g:Z\longrightarrow X$ such that $g^*E$ is trivial. 
\end{enumerate} 

This allows us to state an equivalent version of Theorem \ref{theorem_main}:

\begin{thm}\label{theorem_main'} Let $X$ be a smooth and projective variety over the algebraically closed field $k$. A vector bundle $E$ over $X$ is essentially finite if and only if it satisfies property (TF).
\end{thm} 

\textbf{Remarks:} (a)  Let $f:Y\longrightarrow X$ be a surjective morphism of finite type such that $f^*E\cong\mathcal O_Y^{\oplus r}$. For each $x:\mathrm{Spec}(k)\longrightarrow X$, we pick $y:\mathrm{Spec}(k)\longrightarrow Y$ satisfying $f\circ y=x$. We have $x^*E = y^*f^*E$, hence $x\mapsto \dim_k x^*E=r$ is a constant function on the closed points of $X$. As $X$ is of finite type, it must be a constant function \cite[p. 125, Ex. 5.8(a)]{hartshorne}. We then apply \cite[p. 125, Ex. 5.8(c)]{hartshorne} to conclude that $E$ is a vector bundle. 

(b) Below, see the remarks on page \pageref{a_remark},  we observe that Theorem \ref{theorem_main} is true in characteristic zero even if we only require $X$ to be \emph{normal}. 

(c) If $k$ is the algebraic closure of a finite field, then Theorem \ref{theorem_main} is a direct consequence of Proposition \ref{2} and Maruyama's conjecture proved by Langer \cite{langer_maruyama}. 
  
\subsection{Some notations and conventions}\label{conventions} 

As before, $k$ stands for an algebraically closed field. A variety is an integral and separated scheme of finite type over $k$. 
Here $X$ will always stand for a projective smooth variety over $k$. The dimension of $X$ is denoted by $d$. For any scheme $Y$, the category of vector bundles over $Y$ will be denoted by $\mathbf{VB}(Y)$. By $L$ we will denote a very ample line bundle on $X$. Degrees of vector bundles are always taken with respect to the polarization $L$. The rational number
$$
\mu(E)\, :=\,\frac{c_1(L)^{d-1}\cap c_1(E)}{\text{rank}(E)}=\,\frac{\text{degree}(E)}{\text{rank}(E)}
$$
is called the \textit{slope} of $E$.

\emph{Tannakian categories.} We will follow the conventions of
\cite{deligne-milne}. All Tannakian categories will be over the field 
$k$, and if $\omega:\mathcal T\longrightarrow k\mathrm{-mod}$ is a fiber 
functor, we will denote by $\pi(\mathcal T,\omega)$ the corresponding group scheme
\cite[Theorem 2.11]{deligne-milne}, \cite{nori}.

%%%%%%%%%%%%   TANNAKIAN PROPERTIES  %%%%%%%%%%%%%%%%%%
\section{Tannakian properties}

We use the formalism of Tannakian categories to study property 
(T).

\subsection{The category of objects with property (T)}

\begin{dfn}\label{1}
{\rm (i) Let $\mathcal T(X)$ denote the full subcategory of $\mathbf{VB}(X)$ whose objects are vector
bundles satisfying property (T), or equivalently, property (TF).}

{\rm (ii) Let $f: Y\longrightarrow X$ be a surjective and proper morphism. We will denote by $\mathcal T_Y(X)$ the full subcategory of $\mathcal T(X)$ whose objects become trivial when pulled back to $Y$.} 
\end{dfn}

A vector bundle $E\,\longrightarrow\, X$ is called \textit{Nori--semistable} if, for
every pair $(C\, ,\alpha)$, where $C/k$ is a smooth projective curve and 
$\alpha\, :\, C\,
\longrightarrow\, X$ is a morphism, the pull--back $\alpha^*E$ is semistable of degree zero.

\begin{prp}\label{2} Let $E\in \mathcal T(X)$. Then $E$ is Nori--semistable. 
\end{prp}

\begin{proof}Let $\alpha:C\longrightarrow X$ be a morphism from a curve.
Let $f\, :\, Y\, \longrightarrow\, X$ be a surjective proper morphism 
such that $f^*E$ is trivial. We can 
find a smooth and projective curve $C'/k$ and
morphisms $\beta\, :\, C'\,\longrightarrow\, Y$ and
$\nu:C'\longrightarrow C$ such that $f\circ\beta\,=\,\alpha\circ\nu$ and $\nu$ is surjective.
Hence $\nu^*\alpha^*E$ is trivial. This implies that 
$\alpha^*E$ is semistable of degree zero. 
\end{proof}

\begin{cor}\label{4}(i) Let $f: Y\longrightarrow X$ be a proper 
surjective morphism from a variety $Y$. Then $\mathcal T_Y(X)$ is 
abelian.

(ii)   The category $\mathcal T(X)$ is abelian.

(iii) Let $Y$ be as before, $E,Q$ be vector bundles of degree zero, and 
$\alpha:E\longrightarrow Q$ be an epimorphism. Then, if $f^*E$ is trivial, so is $f^*Q$. In particular, the 
subcategory $\mathcal T_Y(X)$ is stable under quotients in $\mathcal 
T(X)$.
\end{cor}
\begin{proof}(i) Let $\alpha:E\longrightarrow F$ be an arrow of $\mathcal T_Y(X)$. From the assumptions, we conclude that 
$\mathrm{Ker}(f^*(\alpha))$, $\mathrm{Im}(\varphi^*(\alpha))$ and 
$\mathrm{Coker}(f^*(\alpha))$ are all trivial. 
Since $E$ and $F$ are Nori--semistable, it follows that 
$\mathrm{Ker}(\alpha), \mathrm{Im}(\alpha)$ and 
$\mathrm{Coker}(\alpha)$ are all vector bundles \cite[Lemma 3.6]{nori}. 
It is 
then straightforward to see that $f^*$ commutes with $\mathrm{Ker}(\alpha)$ and $\mathrm{Coker}(\alpha)$ so that the kernel and cokernel of $\alpha$ are in $\mathcal T_Y(X)$.

Part (ii) can be easily deduced from (i), since every pair $E\, 
,E'\,\in\, \mathcal T(X)$ belong to $\mathcal T_Y(X)$ for some $Y$ which
is integral.

To prove part (iii),
let $q$ and $n$ be the ranks of $Q$ and $E$ respectively. Let 
$\mathrm{Gr}(q,n)$ 
be the Grassmannian parametrizing linear subspaces of codimension $q$ in 
${\mathbb A}^n$. We let $\mathcal U$ be the universal quotient of $\mathcal 
O_{\mathrm{Gr}(q,n)}^{\oplus n}$, so that the line bundle 
$\mathrm{det}(\mathcal U)\, \longrightarrow\, \mathrm{Gr}(q,n)$ is very 
ample 
\cite{nitsure}.  The hypothesis on $E$ allows us to define a morphism 
\[
\gamma:Y\longrightarrow \mathrm{Gr}(q,n) 
\]  
such that $\gamma^*\mathcal U=f^*Q$. Unless the dimension of the image 
of $\gamma$ is zero, the ampleness of $\mathrm{det}(\mathcal U)$
contradicts the assumption $\deg(Q)=0$. Hence $\gamma$ is a trivial 
morphism, and thereby $f^*Q$ is trivial.
\end{proof}

It is clear that if $E,F$ are objects of $\mathcal T(X)$ (respectively,
$\mathcal T_Y(X)$), then $E\otimes_{\mathcal O_X}F$ is also an object 
of $\mathcal T(X)$ (respectively, $\mathcal T_Y(X)$). This endows 
$\mathcal 
T(X)$ (respectively, $\mathcal T_Y(X)$) with a $k$--linear monoidal 
structure. 

\begin{cor}\label{5} The category $\mathcal T(X)$ is Tannakian over 
$k$. If $x_0$ is a $k$--rational point, then taking the fiber at $x_0$ 
defines an exact and faithful tensor functor $x_0^*\,:\,\mathcal 
T(X)\,\longrightarrow\, k\mbox{--mod}$.
\end{cor}

\subsection{Reformulation of Theorem \ref{theorem_main} in terms of the category $\mathcal 
T(X)$}
In order to state the next result, we need to introduce some terminology and recall some well known results from the theory of Tannakian categories. 
 
\begin{dfn}\label{monodromy} {\rm Let $(\mathcal T,\otimes)$ be a 
Tannakian 
category over $k$ and $V$ an object of $\mathcal T$. The} monodromy 
category of $V$, or the Tannakian subcategory generated by $V$, {\rm is 
the full subcategory}
\[
\langle   V;\mathcal T  \rangle_\otimes
\]
{\rm of $\mathcal T$ admitting as objects the sub-quotients of all 
generalized tensor powers}
\[
\left[V^{\otimes a_1}\otimes (V^{\vee})^{\otimes b_1}\right]\oplus
\cdots \oplus \left[ V^{\otimes a_r}\otimes (V^{\vee})^{\otimes 
b_r}\right]\, ,
\] 
{\rm where $V^\vee$ is the dual to $V$. If $\omega:\mathcal 
T\longrightarrow 
k\mathrm{-mod}$ is a fiber functor, we define the monodromy group of $V$ 
at $\omega$ to be the Tannakian group scheme associated to the monodromy 
category via $\omega$. (See \cite[Theorem 2.11]{deligne-milne}).}
\end{dfn}

\textbf{Remark:}  It is not hard to show that if $\mathcal 
T\,=\,\mathrm{Rep}(G)$, where $G$ is an affine group scheme over $k$, 
then 
the 
monodromy 
category of $V\in\mathrm{Rep}(G)$ is equivalent to the category of representations of $\mathrm{Im}\, (G\longrightarrow \mathrm{GL}(V))$.

\begin{dfn}\label{dfn_of_finiteness} {\rm Let $\mathcal T$ be a 
Tannakian category over $k$. We say that 
$\mathcal T$ is} finite {\rm if there exists an object $\Phi$ such that 
every object $M\in \mathcal T$ is a subquotient of some direct sum 
$\Phi^{\oplus a}$. (The terminology is justified by \cite[Proposition 2.20]{deligne-milne}.)} 
\end{dfn}

The following theorem, which will be proved in Section \ref{sec4}, 
implies Theorem \ref{theorem_main}.

\begin{thm}\label{theorem_main''} For each $E\in\mathcal T(X)$, the Tannakian category $\langle   E;\mathcal T(X)       \rangle_\otimes$ is finite (see Definition \ref{dfn_of_finiteness}). 
\end{thm}
\begin{proof}[Proof that Theorem \ref{theorem_main''} implies Theorem 
\ref{theorem_main}:] Take any $E\in\mathcal T(X)$, and let $G$ be the 
finite group scheme associated, via the fiber functor $x_0^*$, to the 
category $\langle E;\mathcal T(X) \rangle_\otimes$. Hence, by the 
results in \cite[\S~2]{nori}, there exists a $G$-torsor 
$P\longrightarrow X$ such that the functor
\[
P\times^G(\bullet)\,:\, \mathrm{Rep}(G)\,\longrightarrow\, \langle 
E;\mathcal T(X) \rangle_\otimes\, ,
\]
that send any $G$--module $V$ to the associated vector bundle
$P\times^G V$, is an equivalence of Tannakian categories.
But if $V$ is a finite dimensional representation of $G$, then 
\cite[Proposition 3.8]{nori} shows that $P\times^GV$ is essentially 
finite. 
\end{proof}

\section{The S--fundamental group scheme and reduction to the case of 
curves}
We will show how to reduce the proof of Theorem \ref{theorem_main''} to 
the case where $X$ is of dimension one. This is possible due to a 
``Lefschetz Theorem'' proved in \cite{langer_preprint}.

\subsection{The S--fundamental group scheme \cite{BPS}, 
\cite{langer_preprint}}
\begin{dfn}\label{S-sheaves} {\rm The category $\mathrm{Vect}^s_0(X)$ is 
the full sub-category of $\mathbf{VB}(X)$ whose objects are strongly 
semistable vector bundles $V$ with}
\[
[c_1(L)^{d-1}\cap c_1(V)]\, =\, [c_1(L)^{d-2}\cap \mathrm{ch}_2(V)]\, 
=\, 0\, ,
\]
{\rm where $L$ is a fixed polarization on $X$ and $d=\dim\,X$.}
\end{dfn}

The category $\mathrm{Vect}^s_0(X)$ in Definition \ref{S-sheaves} is 
Tannakian \cite[Proposition 5.4]{langer_preprint} and the fiber functor
constructed using a $k$--point $x_0$ defines a fundamental group scheme:
\[
\pi(\mathrm{Vect}^s_0(X),x_0)\, .
\]

Let $\mathrm{Fr}_X\, :\, X\, \longrightarrow\, X$ be the absolute
Frobenius morphism if $\mathrm{char}(k)>0$ and the identity map otherwise.  

If $E\in\mathcal T(X)$, then clearly $\mathrm{Fr}_X^*E$ is also in 
$\mathcal T(X)$. Hence, if $E$ is a vector bundle of $\mathcal T(X)$,
then $(\mathrm{Fr}_X^m)^*E$ is Nori--semistable for every $m\, 
\in\, \mathbb N$; 
this entails that $E$ is strongly semistable with respect to the
polarization $L$ on $X$. Using 
the projection formula for Chern classes, it is also clear that $c_i(E)$ is numerically trivial for any $E\in \mathcal T(X)$ and any $i>0$. Thus we have a natural fully faithful exact functor  of Tannakian categories 
\begin{equation}\label{functor_of_cats}
\mathcal T(X)\longrightarrow \mathrm{Vect}_0^s(X).
\end{equation}
By Corollary \ref{4} (iii) and  \cite[Proposition 2.21]{deligne-milne}, we have:
\begin{lem}\label{application_stability_quotients}
The homomorphism of group schemes corresponding to \eqref{functor_of_cats} 
\begin{equation}\label{surj}
\pi(\mathrm{Vect}^s_0(X),x_0)\, \longrightarrow\, \pi(\mathcal T(X),x_0)
\end{equation}
is faithfully flat. 
\end{lem} 

\begin{prp}\label{curves_implies_general}Assume that Theorem \ref{theorem_main''} holds for curves. Then it holds for higher dimensional $X$. 
\end{prp} 
\begin{proof}We will proceed by induction. We assume that 
$\dim(X)\,\ge\, 
2$ and that the theorem has been proved for all smooth 
projective varieties $C$ with $\dim C\, <\, \dim X$. We will now establish the existence of a smooth irreducible effective divisor $C\hookrightarrow X$ such that 
the natural 
homomorphism of group schemes  
\begin{equation}\label{leff}
\pi(\mathrm{Vect}_0^s(C),x_0)\longrightarrow \pi(\mathrm{Vect}_0^s(X),x_0)
\end{equation}
is faithfully flat. Due to \cite[Theorem 
10.2]{langer_preprint}, it is enough to find a smooth, connected, ample effective divisor $C$ of high degree. We now apply Bertini's Theorem (see \cite[p. 179, Theorem 8.18]{hartshorne} and \cite[7.9.1, p. 245]{hartshorne}) to $X$ embedded in $\mathbb P^N$ using the line bundle $L^{\otimes r}$.

Let 
$f:Y\longrightarrow X$ be a finite morphism from a variety $Y$ to $X$ 
and consider $E\in\mathcal T_Y(X)$. Clearly the restriction $E\vert_C$ 
is an object of $\mathcal T(C)$. From
Lemma \ref{application_stability_quotients} we know that the natural functors 
\[
\langle E; \mathcal T(X)\rangle_\otimes\, \longrightarrow\, \langle E; 
\mathrm{Vect}^s_0(X)  \rangle_\otimes
\]
and 
\[
\langle E|_C; \mathcal T(C)\rangle_\otimes\,\longrightarrow\,\langle 
E|_C; \mathrm{Vect}^s_0(C)  \rangle_\otimes
\]
are equivalences: they are fully faithful, their essential image is stable by subquotients and, by definition (see Definition \ref{monodromy}), any object of the target category is a subquotient of an object from the source category. 
As the homomorphism in eq. \eqref{leff} is faithfully flat, the functor
\[
\langle E; \mathrm{Vect}^s_0(X)  \rangle_\otimes\,\longrightarrow\,  
\langle E|_C; \mathrm{Vect}^s_0(C)\rangle_\otimes 
\]
is also an equivalence (follows by repeating the previous argument). In conclusion, the Tannakian categories
$$
\langle E; \mathcal T(X)\rangle_\otimes ~\,~\,~\,~\,\text{~and~} 
~\,~\,~\,~\,
\langle E|_C; \mathcal T(C)\rangle_\otimes
$$
are equivalent, and hence $\langle E; \mathcal T(X)\rangle_\otimes$ 
is finite by the induction hypothesis.  
\end{proof}

%%%%%%%%%%%%%%%      SEPARABLE     %%%%%%%%%%%%%%%
\section{The case of curve (conclusion of proof)}\label{sec4}

{}From now on we assume 
that $X$ is a smooth projective (connected) curve over $k$. This has the consequence that  torsion free sheaves and vector bundles coincide. 

\subsection{The maximal slope of certain coherent $\mathcal 
O_X$--algebras}
Let 
\[
f\,:\,Y\,\longrightarrow \,X
\] 
be a finite dominant morphism from a projective curve $Y$
and take any $E\in\mathcal T_Y(X)$, i.e.,
the vector bundle $f^*E$ 
is trivial. We assume that the 
extension of function fields provided by $f$ is \emph{separable} (in 
other words, $f$ is generically \'etale). Let $\mathcal A$ denote the 
coherent $\mathcal O_X$--algebra $f_*(\mathcal O_X)$. 

The following very simple observation is the key for all further considerations: Using the projection formula we have an isomorphism 
\begin{equation}
E\otimes_{\mathcal O_X} \mathcal A\cong  \mathcal A^{\oplus r}\, .  
\end{equation}
This isomorphism induces a monomorphism of $\mathcal O_X$--modules 
\begin{equation}
\alpha:E\hookrightarrow \mathcal A^{\oplus r}\, .
\end{equation}

\begin{prp}\label{image_in_max} Let $\mathcal A_\mathrm{max}$ denote the maximal destabilizing subbundle of $\mathcal A$. Then the image of $\alpha$ is contained in $(\mathcal A_\mathrm{max})^{\oplus r}$.
\end{prp}
\begin{proof}Evidently, the proposition will be proved once we 
establish that 
$$
\mu_\mathrm{max}(\mathcal A^{\oplus r})\,=\, \mu_\mathrm{max}(\mathcal 
A)\,=\, 0\, .
$$
Note that it is enough to show that $\mu_\mathrm{max}(\mathcal
A)\,\le\, 0$.

Assume that $\mu_\mathrm{max}\,>\, 0$, so that $\mathcal 
A_\mathrm{max}$ 
is semistable of positive slope. By adjointness, we have that 
\[
\mathrm{Hom}_Y(f^*(\mathcal A_\mathrm{max})\, ,\mathcal O_Y)\, =\,  
\mathrm{Hom}_X(\mathcal A_\mathrm{max},\mathcal A)\,\not=\, 0\, .
\] 
But, by the separability hypothesis made on $f$, we know that 
$f^*(\mathcal A_\mathrm{max})$ is ($f^*L$)--semistable of positive 
degree, so there are no non-zero homomorphisms from $f^*(\mathcal 
A_\mathrm{max})$ to $\mathcal O_Y$. Hence $\mu_\mathrm{max}(\mathcal
A)\,\le\, 0$. 
\end{proof}

Proposition \ref{image_in_max} has the following corollary:

\begin{cor}\label{semistable_dominating} Let $f:Y\longrightarrow X$ be as above. 
Then there exists a semistable locally free coherent $\mathcal O_X$--module 
$M\, :=\,{\mathcal A}_\mathrm{max}$ of degree 
zero such that for each $E\,\in\, \mathcal T_Y(X)$, there exists a 
monomorphism of $\mathcal O_X$--modules 
\[
\alpha_E\,:\, E\,\hookrightarrow \, M^{\oplus \mathrm{rank}\, E}\, .
\]
\end{cor}

\subsection{The case of separable (generically \'etale) 
morphisms}\label{section_main_separable} We continue with the above 
notation:
$$f\,:\,Y\,\longrightarrow \,X$$ is a finite surjective morphism from a 
projective curve $Y$ which induces a \emph{separable extension} of function 
fields; $E$ is a vector bundle on $X$ such that $f^*E$ is 
trivial. 

Recall that in Corollary \ref{semistable_dominating}, we showed that 
each $V\,\in\,\langle E;\mathcal T(X)\rangle_\otimes$ is a sub--quotient 
of 
a direct sum of copies of a fixed torsion free semistable coherent sheaf of 
slope zero.

\begin{thm}\label{theorem_main_separable} Let $f\,:\,Y\,\longrightarrow 
\,X$ 
and $E$ be as above. Then the category $\langle E;\mathcal 
T(X)\rangle_\otimes$ is finite. In particular, Theorem \ref{theorem_main''} (and hence Theorem \ref{theorem_main}) is true if $k$ has characteristic zero (see also the remark below). 
\end{thm}

\begin{proof} As the subcategory $\mathcal T_Y(X)$ of $\mathcal T(X)$ 
is stable under sub--quotients (Proposition \ref{2}), we have $\langle 
E;\mathcal T_Y(X)  \rangle_\otimes\, =\, \langle E;\mathcal 
T(X)\rangle_\otimes$. Let $M$ be the $\mathcal O_X$--module appearing in 
Corollary \ref{semistable_dominating}, so $M\,=\,{\mathcal A}_\mathrm{max}$.
We want to find a vector bundle
$\sigma(M)\,\in\,  \langle E;\mathcal T_Y(X)  \rangle_\otimes$ which is a submodule of $M$ and 
which 
induces, for every $V\,\in\, \langle E;\mathcal T_Y(X)  \rangle_\otimes$, a factorization
\[
\xymatrix{   V\ar@{^{(}->}[r]^{\alpha_V} \ar[d]  &  M^{\oplus l} \\  \sigma(M)^{\oplus l}\ar@{^{(}->}[ru]     }
\]
of the monomorphism $\alpha_V$ displayed in Corollary \ref{semistable_dominating}. By definition, this will prove that  
$\langle E;\mathcal T_Y(X)  \rangle_\otimes$ is finite. 

Now let $N$ be an arbitrary semistable
torsionfree sheaf on $X$ of slope zero. Let $\sigma(N)\,\subseteq \,N$ be the largest sub--object of $N$ belonging to $\langle E; \mathcal T_Y(X)\rangle_\otimes$; the existence of $\sigma(N)$ is guaranteed by the following two facts
\begin{enumerate}
\item each ascending chain of sub--sheaves $N_1\subseteq N_2\subseteq \cdots\subseteq N$ must terminate;
\item if $N_1$ and $N_2$ are sub--objects of $N$ belonging to $\langle E; \mathcal T_Y(X)\rangle_\otimes$, then $$N_1+N_2\,=\,\mathrm{Im}(N_1\oplus N_2\longrightarrow N)$$ must also be in $\langle E; \mathcal T_Y(X)\rangle_\otimes$, due to Proposition \ref{2}. 
\end{enumerate}
Let $V\in\langle E; \mathcal T_Y(X)\rangle_\otimes$, and let $\alpha_V\,:\,V\,\hookrightarrow \, M^{\oplus l}$ be a monomorphism. It follows that $\alpha_V$ factors through the inclusion $\sigma(M^{\oplus l})\,\subseteq\, M^{\oplus l}$. But again, using Proposition \ref{2}, we see that $\sigma(M^{\oplus l})\,=\, \sigma(M)^{\oplus l}$. 
\end{proof}

\textbf{Remarks:} (1) In the proof of Theorem \ref{theorem_main_separable}, we considered the largest sub--object lying in $\mathcal T_Y(X)$ of a torsionfree semistable sheaf of slope zero. This can be put in a more abstract setting: finding a right adjoint for the inclusion of $\mathcal T_Y(X)$ into the category of torsionfree semistables of slope zero. The important point is, of course, stability
under quotients (Proposition \ref{2}). The reasoning is reminiscent of the construction of a right adjoint for the inclusion of categories $\mathrm{Rep}(H)\longrightarrow \mathrm{Rep}(G)$, where $G\longrightarrow H$ is surjective. 

\label{a_remark}(2) In characteristic zero, there is also an easy proof of Theorem \ref{theorem_main} which only assumes that the existence of a \emph{trace morphism} 
\[
\mathrm{Tr}_{\mathcal A/\mathcal O_X}:\mathcal A\, \longrightarrow \,\mathcal O_X\, .
\]
(So normality of $X$ is already sufficient.)
Such an $\mathcal O_X$--linear morphism allows us to find a section of the inclusion of $\mathcal O_X$ modules $\mathcal O_X\hookrightarrow \mathcal A$, so that, for $E\in\mathcal T_Y(X)$, each $E^{\otimes n}$ is a direct summand of $\mathcal A^{\oplus l}$. Hence, the indecomposable coherent $\mathcal O_X$--modules appearing in $E^{\otimes n}$ are isomorphic to certain indecomposable components of $\mathcal A$ (this uses the uniqueness of the Remak decomposition, see \cite[p. 313, Theorem 1]{At} and \cite[p. 315, Theorem 2]{At}); we then apply \cite[Lemma 3.1]{nori} to conclude that $E$ is \emph{finite}.

\subsection{Proof of Theorem \ref{theorem_main''} in the case of curves}\label{the_case_of_curves}

Let $E\in\mathcal T_Y(X)$, where $f:Y\longrightarrow X$ is a finite surjective
morphism. We can assume without loss of generality that $Y$ is smooth and irreducible. Using the fact that the only purely inseparable morphisms between smooth curves are the Frobenia,  we can find a factorization of $f$ as $$f\,=\, \mathrm{Fr}_X^{m}\circ g\, ,$$ where $g:Y\longrightarrow X$ induces a separable extension of function fields.
Then 
\[
\langle (\mathrm{Fr}_X^m)^*E ; \mathcal T(X)\rangle_\otimes
\] 
is finite due to Theorem \ref{theorem_main_separable}. Let $G$ be the monodromy group of $E$ in $\mathcal T(X)$ (see Definition \ref{monodromy}), and let 
\[
\rho\,:\,G\,\longrightarrow\, \mathrm{GL}(x_0^*E)\,=\,\mathrm{GL}_{r}
\] 
be the \emph{faithful} monodromy representation. Denote by
$$\varphi:k\longrightarrow k$$ the arithmetic Frobenius $a\longmapsto a^{p^m}$ and by $\rho^{(m)}$ the twist of $\rho$ by $\varphi$; in concrete terms: if $\rho$ has matrix coefficients given by $(\rho_{ij})\in \mathrm{GL}_r(\mathcal O(G))$, then the matrix coefficients of $\rho^{(m)}$ are $(\rho_{ij}^{p^m})$. Let $G^{(m)}$ be the $k$--group scheme $G\otimes_{k,\varphi}k$; this is the scheme $G$ endowed with a different morphism to $\mathrm{Spec}\,(k)$.  We have a commutative diagram of homomorphisms of $k$-group schemes
\[
\xymatrix{ G\ar[dr]_{\rho^{(m)}}\ar[r]^{\mathrm{Fr}_G^m}   & G^{(m)} \ar[d]^\gamma \\ & \mathrm{GL}_r       }
\]  
where $\gamma$ is defined by the matrix coefficients $(\rho_{ij})$, now regarded as elements in $\mathcal O(G^{(m)})$. It follows that $\mathrm{Ker}(\gamma)=\{1\}$, while $\mathrm{Ker}(\mathrm{Fr}_G^m)$ is a finite local group scheme (its ring of functions is a local Artin algebra). 
Since the representation 
\[
\xymatrix{\pi(\mathcal T(X),x_0)\ar[r] & G \ar[r]^{\rho^{(m)}}   & \mathrm{GL}_r           }  
\] 
corresponds to $(\mathrm{Fr}_X^m)^*E$, the image of $\rho^{(m)}$ in $\mathrm{GL}_r$ is finite (recall that the first arrow above is faithfully flat). Hence, $G$ is an extension of finite group schemes, which shows that $G$ is a finite group scheme. We have proved Theorem \ref{theorem_main''} for curves and hence (see Proposition \ref{curves_implies_general}) for any smooth projective variety. 
 
%%%%%%%%%%%%%%%%%%%%%%%%%%%%%%%%%%%%%%%%

\end{document}